\documentclass[11pt]{amsart}
\usepackage{amssymb, latexsym, graphicx, mathrsfs, enumerate, xcolor}

\usepackage{manfnt} 
\usepackage{amssymb,amsmath}
\makeindex

\newtheorem{thm}{Theorem}[section]
\newtheorem{cor}[thm]{Corollary}
\newtheorem{lem}[thm]{Lemma}
\newtheorem{prop}[thm]{Proposition}

\theoremstyle{definition}

\theoremstyle{remark}

\newtheorem{conj}[thm]{Conjecture}
\numberwithin{equation}{section}

\newcommand*\cube{\mbox{\mancube}}

\newcommand*\phih{\widehat{\phi}}

\newcommand*{\infint}{\int_{-\infty}^{\infty}}
\makeatletter
\newcommand{\alignfootnote}[1]{%
	\ifmeasuring@
	\else
	\footnote{#1}%
	\fi
}
\makeatother
\begin{document}
\sloppy 
\title{Low-lying zeros of cubic Dirichlet $L$-functions and the Ratios Conjecture}

\author[Peter J. Cho]{Peter J. Cho$^{\dagger}$}
\thanks{$^{\dagger}$ This work was supported by Basic Science Research Program through the National Research Foundation of Korea(NRF) funded by the Ministry of Education(2016R1D1A1B03935186).}
\address{Department of Mathematical Sciences, Ulsan National Institute of Science and Technology, Ulsan, Korea}
\email{petercho@unist.ac.kr}

\author[Jeongho Park]{Jeongho Park$^{\star}$}
\thanks{$^{\star}$ The author was supported by the National Research Foundation of Korea(NRF) funded by the Ministry of Education, under the Basic Science Research Program(2017R1D1A1B03028670).}
\address{Department of Mathematical Sciences, Ulsan National Institute of Science and Technology, Ulsan, Korea}
\email{pkskng@unist.ac.kr}

\begin{abstract} We compute the one-level density for the family of cubic Dirichlet $L$-functions when the support of the Fourier transform of the associated test function is in $(-1,1)$. We also establish the Ratios conjecture prediction for the one-level density for this family, and confirm that it agrees with the one-level density we obtain. 
\end{abstract}

\maketitle

\section{Introduction}

After the monumental work of Montgomery \cite{MONT}, number theorists have tried to understand the zeros of automorphic $L$-functions via random matrix theory. Katz and Sarnak \cite{KS} proposed a conjecture, which claims that the distributions of low-lying zeros of the $L$-functions in a family $\frak{F}$ is governed by its corresponding symmetry type $G(\frak{F})$. We refer to \cite{KS2} as a kind introduction to the conjecture. For various families, the conjecture has been tested and all the results have supported it. Since it is impossible to give a complete list, we name just a few of them \cite{AM, Yoonbok2017, CK-imrn2015, ILS, Miller, Rub, Young}. Those who are interested in this problem may take a look at the references therein. 
 However, it seems out of reach to prove the conjecture fully. In this sense, it is meaningful to investigate the distributions of low-lying zeros of $L$-functions in a new family even if the result is limited. 

In this work, we study the low-lying zeros of cubic Dirichlet $L$-functions. Let $\phi$ be an even Schwartz function whose Fourier transform is compactly supported. For a cubic Dirichlet character $\chi$, let $\rho$ denote the nontrivial zeros of $L(\chi,s)$ in the critical strip. Define $$D_X(\chi;\phi) = \sum_{\gamma} \phi\left(\gamma \frac{L}{2\pi}\right),$$ where $\gamma = -i(\rho - 1/2)$ and $L=\log \left( \frac{X}{2\pi e} \right)$. Let $w(t)$ be an even Schwartz function that is nonnegative and is nonzero; in particular, at $s=1$, its Mellin transform $$\mathfrak{w}(s) = \int_0^{\infty} w(t) t^{s-1} dt$$ is positive and is differentiable. Let $\omega = \frac{-1+\sqrt{-3}}{2}$. The total weight is defined by $$W^{\ast}(X) = \sum_{\chi}^{\ast} w(q/X) = \sum_{\alpha}''' w\left(\frac{N(\alpha)}{X}\right)$$ where $\sum_{\chi}^{\ast} = \sum_{\alpha}'''$\index{$\sum_{\chi}^{\ast} = \sum_{\alpha}'''$} is the sum over primitive cubic characters $\chi$ parameterized by $\alpha \in \mathbb{Z}[\omega]$, $\alpha \equiv 1$ mod $3$, $\alpha$ is square-free and has no rational prime divisor as in \cite[Lemma 2.1]{BaierandYoung2010}, and $q = q(\chi) =  N(\alpha) > 1$ is the conductor of $\chi$. The 1-level density we are interested in is $$\mathcal{D}^{\ast}(\phi;X) = \frac{1}{W^{\ast}(X)} \sum_{\chi}^{\ast} w\left(\frac{q}{X}\right) D_X(\chi;\phi).$$

In this work we reserve $\sigma$ for $\sup(supp(\phih))$, and $\sum_{\alpha}$ will denote the sum over $\alpha \in \mathbb{Z}[\omega]$. $k,m,n,\ell$ denote natural numbers, and $p$ represents a rational prime. Throughout this paper, we let $$a(n) = \prod_{p|n} a(p),\quad \text{where}\quad a(p) = \begin{cases}
\frac{p}{p+2}\;\; &\text{ if $p \equiv 1$ mod $3$,}\\
1\;\;&\text{ otherwise.}
\end{cases}$$

First, we compute the one-level density $ \mathcal{D}^{\ast}(\phi;X)$ under GRH.

\begin{thm}\label{thm-main}
Assume GRH. \footnote{We need GRH for the following families of $L$-functions: $\zeta(s)$, real quadratic Dirichlet $L$-functions, Hecke $L$-functions over $K = \mathbb{Q}(\sqrt{-3})$ with cubic characters, and $\zeta_K(s)$.} For $\sigma <1$, we have
\begin{align*}
   \mathcal{D}^{\ast}(\phi;X) 
        &= \frac{\phih(0)}{L W^{\ast}(X)} \sum_{\chi}^{\ast} w\left(\frac{q}{X}\right) \log q -
        \frac{\phih(0)\log \pi}{L} \\
        & \qquad - \frac{2}{L} \infint \phi(\tau) \sum_{\ell \geq 1} \sum_p \frac{a(p)\log p}{p^{3\ell/2 + 6\pi i \tau \ell/L}} d\tau\\
        &\qquad + \frac{1}{2L}\infint \phi(\tau) \left(\Psi\left(\frac{1}{4}-\frac{\pi i \tau}{L}\right) + \Psi\left(\frac{1}{4}+\frac{\pi i \tau}{L}\right)\right) d\tau\\
        &\qquad + O\left(X^{-1/2 + \sigma/2 + \epsilon}\right),
\end{align*}
where $\Psi(z)=\frac{\Gamma'(z)}{\Gamma(z)}$ is the digamma function.
\end{thm}

Now, the one level density can be approximated by a polynomial in $1/L$ as follows.
\begin{cor}\label{cor-main}
    Assume GRH. For $\sigma<1$ and any $M \geq 1$, we have
    \begin{align*}
    \mathcal{D}^{\ast}(\phi;X) 
    &= \phih(0) + \frac{\phih(0)}{L}\left(\frac{\mathfrak{w}'(1)}{\mathfrak{w}(1)} + 1 - \gamma - 2\log 2 - \frac{\pi}{2}  - \sum_{\ell\geq 1}\sum_p \frac{2a(p) \log p}{p^{3\ell/2}}\right)\\
    & + \sum_{k=1}^{M-1} \frac{I_k \phih^{(k)}(0)}{k! L^{k+1}} + O_{M,\phi}\left(\frac{1}{L^{M+1}} + X^{-1/2+\sigma/2+\epsilon}\right),
    \end{align*}
    where $$I_k = - \sum_{\ell\geq 1}\sum_p \frac{2 a(p) (3\ell)^k (\log p)^{k+1}}{p^{3\ell/2}} - 2^{k+2}\pi^{k+1}\int_0^{\infty} \frac{x^k e^{-\pi x}}{1-e^{-4\pi x}} dx.$$
\end{cor}

Now we can see that the possible symmetry type for the family of cubic Dirichlet $L$-functions is $U$.
\begin{cor}
Assume GRH. For $\sigma  < 1$, we have
\begin{eqnarray*}
\lim_{X \rightarrow \infty} \mathcal{D}^{\ast}(\phi;X)= \hat{\phi}(0).
\end{eqnarray*}
\end{cor}

Conrey, Farmer and Zimbauer \cite{CFZ} developed the Ratios conjecture,  which is a powerful recipe that predicts various statistics regarding $L$-functions. The Ratios conjecture is applied to compute one-level density for many different families of $L$-functions \cite{FPS2016, FM, HMM, Miller2}. The Ratios conjecture reveals lower order terms which the Katz-Sarnak conjecture is silent about. A typical recipe from this perspective leads us to the following prediction, which agrees with the one-level density in Theorem \ref{thm-main} up to $O(X^{-1/2+\sigma/2+\epsilon})$. 

\begin{thm}[One level density prediction via the Ratios conjecture]\label{thm-ratios conj prediction} Assume Conjecture \ref{conj Rw}. 
Let $\phi$ be an even Schwartz function whose Fourier transform is compactly supported. 

Then, the one-level density for the family of primitive cubic Dirichlet $L$-functions $L(\chi,s)$ is expressed as follows.
    \begin{align*}
        \mathcal{D}^{\ast}(\phi;X) 
        &= \frac{\phih(0)}{L W^{\ast}(X)} \sum_{\chi}^{\ast} w\left(\frac{q}{X}\right) \log q -
        \frac{\phih(0)\log \pi}{L} \\
        & \qquad - \frac{2}{L} \infint \phi(\tau) \sum_{\ell \geq 1} \sum_p \frac{a(p)\log p}{p^{3\ell/2 + 6\pi i \tau \ell/L}} d\tau\\
        &\qquad + \frac{1}{2L}\infint \phi(\tau) \left(\Psi\left(\frac{1}{4}-\frac{\pi i \tau}{L}\right) + \Psi\left(\frac{1}{4}+\frac{\pi i \tau}{L}\right)\right) d\tau\\
        &\qquad + O\left(X^{-1/2 + \epsilon}\right),
    \end{align*}
where $\Psi(z)=\frac{\Gamma'(z)}{\Gamma(z)}$ is the digamma function.
\end{thm}

Fiorilli and Miller \cite{FM} considered a family of all Dirichlet $L$-functions. They found a term which the Ratios conjecture failed to predict when $\sup(supp(\phih)) > 1$. It would be interesting if we were able to expand the support and find such a term. 

This work is inspired by a recent result of Fiorilli, Parks and Sodergren \cite{FPS2017}, which studied one-level density for quadratic Dirichlet $L$-functions for  $supp( \phih) \subset (-2,2)$. We tried to obtain one-level density for $\sup(supp(\phih)) > 1$ but were not able to do. The family of cubic Dirichlet $L$-functions differs from that of quadratic Dirichlet $L$-functions in the following sense:
\begin{enumerate}
    \item The characters are parameterized by the elements of $\mathbb{Z}[\omega]$ which, as a lattice, is of rank 2.
    \item The characters are not self dual.
    \item The phase distribution of the cubic Gauss sums (or the Kummer sums) cannot be given in a simple congruence condition \cite{Heath-Brown1979}.
\end{enumerate}
Whereas (2), (3) do not affect our argument much, (1) is the major obstacle in expanding the support of $\phih(x)$. Cubic Dirichlet $L$-functions of conductor $ < X$, like quadratic ones, form a relatively thin subset of cardinality $O(X)$ among Dirichlet $L$-functions of conductor $< X$ (whose cardinality is $\gg X^2$). From the analytic point of view, this sparsity is an obstacle in expanding the support of the Fourier transform of the test function. The family of quadratic $L$-functions is fortunately parameterized by $\mathbb{Z}$ which is of rank $1$, and this facilitates additional savings to consider wider supports. On the other hand, (1) implies that we do not have such an additional saving for cubic Dirichlet characters. It forces the estimation of $S_1(X) =\sum_p \frac{\log p}{\sqrt{p}} \phih \left( \frac{\log p}{L}\right) \sum_{\chi}^*\omega\left( \frac{q}{X} \right)\chi(p)$ which is the source of the error term $O(X^{-1/2+\sigma/2+\epsilon})$ in Theorem~\ref{thm-main}, to be inferior to that for quadratic $L$-functions when the support of $\phih$ is expanded. 

It also worths mentioning that the cubic Dirichlet characters are given as a restriction of cubic Hecke characters. This affects the strength of the large sieve estimates by Heath-Brown \cite{Heath-Brown2000}, since the summatory variables we take in the large sieve inequality become even thinner. In many cases, this is another difficulty in making estimates involving cubic Dirichlet characters as sharp as the ones for quadratic characters.

In Appendix, we collect some prerequisites which were useful for this work.

\bigskip
\section{One-level Density}

First, we introduce Weil's Explicit Formula which is one of the main tools for one-level density problems. 
\begin{lem}[Weil's Explicit Formula]\label{lem explicit formula}
    \begin{align*}
    \mathcal{D}^{\ast}(\phi;X) &= \frac{\phih(0)}{L W^{\ast}(X)} \sum_{\chi}^{\ast} w\left(\frac{q}{X}\right) \log q - \frac{\phih(0)}{L} \left( \gamma + 3\log 2 + \frac{\pi}{2} + \log \pi\right) \\
    &\qquad - \frac{1}{L W^{\ast}(X)} \sum_{p,m>0} \frac{\log p}{p^{m/2}} \phih\left(\frac{m\log p}{L}\right) \sum_{\chi}^{\ast} w\left(\frac{q}{X}\right) \left( \chi(p^m) + \overline{\chi}(p^m)\right)\\
    &\qquad + \frac{4\pi}{L} \int_0^{\infty} \frac{e^{-\pi x}}{1 - e^{-4\pi x}} \left(\phih(0) - \phih\left(\frac{2\pi x}{L}\right) \right) dx,
    \end{align*}
    where $\gamma = 0.5772156649\cdots$ is the Euler-Mascheroni constant.
\end{lem}
\begin{proof}
    The proof of \cite[Lemma 2.1]{FPS2017} works the same. The only difference is that $\chi$ is always even in our case, because $\chi(-1) = \chi((-1)^3) =1$. The Euler-Mascheroni constant appears from $\frac{\Gamma'}{\Gamma}\left(\frac{3}{4}\right) = -\gamma -\log 8 + \frac{\pi}{2}$. \cite[Theorem 12.13]{montgomery2007multiplicative} 
\end{proof}

To compute the one-level density for our family, we need to know the average of the first term and the third term of Lemma \ref{lem explicit formula}. These parts are carried out in Sec. \ref{sec-first term} to Sec. \ref{sec-third term}. In Sec. \ref{sec-fourth term}, we find a different expression for the fourth term and prove Theorem~\ref{thm-main} and Corollary~\ref{cor-main}.

\subsection{The first term in Lemma~\ref{lem explicit formula}: $\frac{\phih(0)}{L W^{\ast}(X)} \sum_{\chi}^{\ast} w\left(\frac{q}{X}\right) \log q$}\label{sec-first term}
To estimate the first term in Lemma~\ref{lem explicit formula} we compute the sum $\sum_{\alpha}''' w\left(\frac{N(\alpha)}{X}\right) \log N(\alpha)$, for which we follow the proof of \cite[Lemma 2.5 and 2.8]{FPS2016}. In this section we prove the following lemma.
\begin{lem}[The first term of the explicit formula]\label{lem-first term of explicit formula}
Under GRH,
\begin{equation*}
    \sum_{\chi}^{\ast} w\left(\frac{q}{X}\right) \log q 
    = L W^*(X) + \left(\frac{\mathfrak{w}'(1)}{\mathfrak{w}(1)} + \log 2\pi e\right) W^*(X) + O\left(X^{1/3}\log X\right).
\end{equation*}
\end{lem}

We state a proposition before we prove the above lemma.
\begin{prop}\label{prop-product of L functions}
    For Dirichlet characters $\chi_1,\cdots, \chi_r$, we have
    \begin{multline*}
    \prod_{p}\left(1+\sum_{i = 1}^r \frac{\chi_i(p)}{p^s}\right) =
    \prod_i L(\chi_i,s) \cdot \prod_i L(\chi_i^2,2s)^{-1} \cdot \prod_{i<j} L(\chi_i\chi_j,2s)^{-1}
    \cdot \prod_{i<j<k} L(\chi_i \chi_j \chi_k, 3s)^2 \\
    \cdot \prod_{i \neq j} L(\chi_i^2 \chi_j, 3s) \cdot \prod_i L(\chi_i^4,4s) \cdot \prod_{i<j} L(\chi_i^2\chi_j^2,4s)
    \cdot H(\chi_1,\cdots,\chi_r;s),
    \end{multline*}
    where $H(\chi_1,\cdots,\chi_r;s) = \prod_p \left(1 + O\left(\frac{1}{p^{4s}}\right) \right)$ converges absolutely for $Re(s) > 1/4$.
\end{prop}
\begin{proof}
    Writing $x_i = \frac{\chi_i(p)}{p^s}$, it suffices to compute
    \begin{multline*}
    \left(1 + \sum_{i=1}^r x_i \right) \prod_i (1-x_i)\cdot \prod_i \frac{1 - x_i^4}{1-x_i^2} \cdot \prod_{i<j} \frac{1 - x_i^2 x_j^2}{1-x_i x_j} \cdot \prod_{i < j < k}(1-x_i x_j x_k)^2 \cdot \prod_{j < k} (1-x_j^2 x_k)(1-x_j x_k^2) \\
    = 1 + f(x_1,\cdots, x_r),
    \end{multline*}
    where $f(x_1, \cdots, x_r)$ is a polynomial such that every term has degree at least 4.
\end{proof}

\emph{Proof of Lemma~\ref{lem-first term of explicit formula}.}
Let $\nu : \mathbb{Z}[\omega] \rightarrow \{0,1\}$ be defined by $\nu(\alpha) = 1$ if $\alpha \equiv 1$ mod $3$, $\alpha$ is square-free and has no rational prime divisor, and $\nu(\alpha) = 0$ otherwise. Then $$1+\sum_{\alpha}''' \frac{1}{N(\alpha)^s} = \sum_{\alpha} \frac{\nu(\alpha)}{N(\alpha)^s} = \prod_{p = \mathfrak{p}\mathfrak{p}':\text{ split}} \left( 1 + \frac{1}{N(\mathfrak{p})^s} + \frac{1}{N(\mathfrak{p}')^s} \right) = \prod_{p \equiv 1\mod 3} \left(1 + \frac{2}{p^s}\right).$$

Let $\xi$\index{$\xi$} be the real Dirichlet character modulo 3, given by $\xi(n) = \left(\frac{n}{3}\right)$. For $p \equiv 1$ mod 3, we observe that $1 + \frac{2}{p^s} = 1 + \frac{1}{p^s} + \frac{\xi(p)}{p^s}$, and for $p \equiv 2$ mod $3$, $1 = 1 + \frac{1}{p^s} + \frac{\xi(p)}{p^s}$. We thus write
\begin{equation}
	\prod_{p \equiv 1\mod 3} \left(1 + \frac{2}{p^s}\right) = \prod_{p \neq 3} \left(1 + \frac{1}{p^s} + \frac{\xi(p)}{p^s}\right).
\end{equation}
By Proposition~\ref{prop-product of L functions} we can write
\begin{equation}\label{eqn_I(s)}
	\prod_{p \equiv 1\mod 3} \left(1 + \frac{2}{p^s}\right) = \zeta(s) L(\xi,s) \zeta(2s)^{-2}L(\xi,2s)^{-1} \zeta(3s) L(\xi,3s) J(s) =: I(s),
\end{equation}
where $J(s)$ converges absolutely for $Re(s) > 1/4$. We thus consider $I(s)$\index{$I(s)$} as the analytic continuation of $1 + \sum_{\alpha}''' \frac{1}{N(\alpha)^s}$.

Using the Mellin transform identity(Mellin inversion) we write
\begin{equation}\label{eqn Tag-2}
\sum_{\alpha}''' w\left(\frac{N(\alpha)}{X}\right) \log N(\alpha) = \frac{1}{2\pi i} \int_{(2)} \sum_{\alpha} \frac{\nu(\alpha) \log N(\alpha)}{N(\alpha)^s} X^s \mathfrak{w}(s) ds
= -\frac{1}{2\pi i} \int_{(2)} \frac{\partial I(s)}{\partial s} X^s \mathfrak{w}(s) ds.
\end{equation}

We want to move the contour of integration in \eqref{eqn Tag-2} to the left. 
Under GRH, we see that the singularities of $I(s)$ in the region $Re(s) > 1/4$ are at $s = 1$ and $s = 1/3$, which are both simple poles. The derivative $\frac{\partial}{\partial s} I(s)$ then has singularities of order 2 at $s = 1$ and $s = 1/3$, and we can use Cauchy's differentiation formula for these points. The residue at $s=1$ is $$-\lim_{s\rightarrow 1} \left(I'(s) (s-1)^2 \mathfrak{w}(s)\right) X\log X - \lim_{s\rightarrow 1} \left( \frac{\partial}{\partial s}\left( I'(s) (s-1)^2 \mathfrak{w}(s)\right)\right) X.$$ 
It is easy to see that $$-\lim_{s\rightarrow 1} \left(I'(s) (s-1)^2 \mathfrak{w}(s)\right)  = \underset{s=1}{Res}\left( I(s)\mathfrak{w}(s)\right),$$ and $$-\lim_{s\rightarrow 1} \left( \frac{\partial}{\partial s}\left( I'(s) (s-1)^2 \mathfrak{w}(s)\right)\right) = \underset{s=1}{Res}\left( I(s) \mathfrak{w}'(s)\right) = \frac{\mathfrak{w}'(1)}{\mathfrak{w}(1)} \underset{s=1}{Res} \left( I(s)\mathfrak{w}(s)\right).$$

We move the contour of integration to $(1/4+\epsilon)$. By Lemma \ref{bound lemma1}, Lemma \ref{bound lemma2} and the fast decaying property of $\mathfrak{w}$ (\cite[Lemma 2.1]{FPS2016}), we have
\begin{equation*}
\sum_{\alpha}''' w\left(\frac{N(\alpha)}{X}\right) \log N(\alpha) = \underset{s=1}{Res}\left( I(s) \mathfrak{w}(s)\right) \left(X\log X + \frac{\mathfrak{w}'(1)}{\mathfrak{w}(1)} X\right) + O\left(X^{1/3} \log X\right). 
\end{equation*}
The same argument implies that $$W^*(X) = \sum_{\alpha}''' w\left(\frac{N(\alpha)}{X}\right) = \underset{s=1}{Res}\left(I(s)\mathfrak{w}(s)\right) X + O\left(X^{1/3}\right),$$ or $$\underset{s=1}{Res}\left(I(s)\mathfrak{w}(s)\right) = \frac{W^*(X)}{X} + O\left(X^{-2/3}\right).$$
Hence we have Lemma~\ref{lem-first term of explicit formula}.
\qed

\subsection{The third term of Lemma~\ref{lem explicit formula}: $\frac{1}{L W^{\ast}(X)} \sum_{p,m} \frac{\log p}{p^{m/2}} \phih\left(\frac{m\log p}{L}\right) \sum_{\chi}^{\ast} w\left(\frac{q}{X}\right) \left( \chi(p^m) + \overline{\chi}(p^m)\right)$}\label{sec-third term}

In this section we prove the following lemma.
\begin{lem}[The third term of the explicit formula]\label{lem-third term of explicit formula}
    Under GRH, for any fixed $M \geq 1$, 
\begin{align*}
     \sum_{p,m>0} \frac{\log p}{p^{m/2}} \phih\left(\frac{m\log p}{L}\right) &\sum_{\chi}^{\ast} w\left(\frac{q}{X}\right) \left( \chi(p^m) + \overline{\chi}(p^m)\right) \\
     &= W^{\ast}(X) \sum_{\ell \geq 1}\sum_p \frac{2 a(p)\log p}{p^{3\ell/2}}\phih\left(\frac{3\ell \log p}{L}\right) +  O\left(X^{1/2 + \sigma/2 + \epsilon}\right) \\
     &= W^{\ast}(X) \sum_{k=0}^{M-1} \left(\frac{\phih^{(k)}(0)}{k! L^k} \sum_{\ell \geq 1}\sum_p \frac{2a(p) (3\ell)^k (\log p)^{k+1}}{p^{3\ell/2}} \right) \\
     &\qquad \qquad + O_{M,\phi}\left( \frac{W^{\ast}(X)}{L^M }+ X^{1/2 + \sigma/2 + \epsilon}\right).
\end{align*}
\end{lem}

Let $$S_m(X) = \sum_{p} \frac{\log p}{p^{m/2}} \phih\left(\frac{m\log p}{L}\right) \sum_{\chi}^{\ast} w\left(\frac{q}{X}\right) \chi(p^m), \qquad \kappa(n) = \prod_{p|n}p,$$\index{$\kappa(n)$} and for $n>0$, $\psi_n(\alpha) = \left(\frac{n}{\alpha}\right)_3$\index{$\psi_n(\alpha) = \left(\frac{n}{\alpha}\right)_3$} ($n=3$ is possible).  We start with following estimations.
\begin{lem}\label{lem Tag-5} Under GRH,
   for $n=$a cube, $n>0$,
    \begin{eqnarray*}
    \sum_{\alpha}''' w\left(\frac{N(\alpha)}{X}\right) \psi_n(\alpha)
    &=& a(n) W^{\ast}(X) + O\left(X^{1/3}\right).
    \end{eqnarray*}
\end{lem}

\begin{proof}
Since $n$ is a cube, for $(n,N(\alpha))=1$, $\psi_n(\alpha)=\left( \frac{n}{\alpha} \right)_3=1$ and $\psi_n(\alpha)=0$ otherwise. As before, we write
    \begin{equation}\label{eqn Tag-5}
    w\left(\frac{1}{X}\right) + \sum_{\alpha}''' w\left(\frac{N(\alpha)}{X}\right) \psi_n(\alpha) = \frac{1}{2\pi i}\int_{(2)} \sum_{\alpha} \frac{\nu(\alpha)\psi_n(\alpha)}{N(\alpha)^s} X^s \mathfrak{w}(s) ds.
    \end{equation}
    In case $\mbox{$n=$a cube}$, the sum over $\alpha$ in the integral reduces to $\sum_{\substack{\alpha\\ (\alpha,n)=1}}\frac{\nu(\alpha)}{N(\alpha)^s} = I(s) \prod_{\substack{p | n\\ p\equiv 1\mod 3}} \frac{p^s}{p^s+2}$ where $I(s)$ is given in \eqref{eqn_I(s)}, and we obtain an analogue of Lemma~\ref{lem-first term of explicit formula}. In this case, since the integrand does not involve a double pole, there are no derivative terms and hence no $\log X$ term appears. The case $n=1$ gives $W^{\ast}(X)$, so the result can be written in terms of $W^{\ast}(X)$.
\end{proof}

\begin{lem}\label{lem Tag-6}
    Under GRH, for $\mbox{$n \neq$a cube}$, $n>0$,
    \begin{equation}
    \sum_{\alpha}''' w\left(\frac{N(\alpha)}{X}\right) \psi_n(\alpha) \ll_{w,\epsilon} \kappa(n)^{\epsilon} X^{1/2 + \epsilon}.
    \end{equation}
\end{lem}
\begin{proof}
    Assume $\mbox{$n \neq$ a cube}$. When $\mathfrak{p} = (\alpha)$ with a primary $\alpha$, we also write $\psi_n(\mathfrak{p})$ for $\psi_n(\alpha)$. By the property of the cubic residue symbol, we note that for any $p \equiv 1$ mod $3$ splitting into $\mathfrak{p}\mathfrak{p}'$, $\psi_n(\mathfrak{p}') = \overline{\psi}_n(\mathfrak{p})$, and for any $p \equiv 2$ mod $3$, $\psi_n(\mathfrak{p}) = 1$. Then, generalizing Proposition~\ref{prop-product of L functions} to the Dedekind zeta function of $K = \mathbb{Q}(\omega)$ and Hecke characters for $K$, we write
    \begin{multline*}
    \sum_{\alpha} \frac{\nu(\alpha) \psi_n(\alpha)}{N(\alpha)^s}
    = \prod_{\substack{p \equiv 1\mod 3\\p = \mathfrak{p}\mathfrak{p}'}} \left(1 + \frac{\psi_n(\mathfrak{p})}{N(\frak{p})^s} + \frac{\overline{\psi}_n(\mathfrak{p})}{N(\frak{p})^s}\right)\\
    = \prod_{\mathfrak{p}\nmid 3} \left(1 + \frac{\psi_n(\mathfrak{p})}{N(\mathfrak{p})^s} + \frac{\overline{\psi}_n(\mathfrak{p})}{N(\mathfrak{p})^s}\right)^{1/2}\cdot \prod_{ p\equiv 2\mod 3}\left( 1+ \frac{2}{p^{2s}}
\right)^{-1/2}\\
    = L_K(\psi_n,s)^{1/2} L_K(\overline{\psi}_n,s)^{1/2} L_K(\overline{\psi}_n,2s)^{-1/2} L_K(\psi_n,2s)^{-1/2} 
\zeta_K(2s)^{-1/2}\\
    L_K(\psi_n,3s)^{1/2} L_K(\overline{\psi}_n,3s)^{1/2} E(s)
    \end{multline*}
    for some $E(s) = \prod_{ p\equiv 2 \mod 3}\left(1+\frac{2}{p^{2s}} \right)^{-1/2}\prod_{\mathfrak{p}} \left(1 + O(N(\mathfrak{p})^{-4s})\right)^{1/2}$ which converges absolutely for $Re(s) > 1/2$. Hence, under GRH, the series $\sum_{\alpha}\frac{\nu(\alpha)\psi_n(\alpha)}{N(\alpha)^s}$ has analytic continuation to the region $Re(s) > 1/2$. As in Lemma~\ref{lem-first term of explicit formula}, we move the contour of integration in \eqref{eqn Tag-5} from $(2)$ to $(1/2 + \epsilon)$.\footnote{Under GRH, $L_K(\psi_n, s)^{1/2}$ has branching points at $Re(s) = 1/2$.} By Lemma \ref{bound lemma2} and the fast decaying property of $\mathfrak{w}(s)$, we obtain the lemma.
\end{proof}

We now are about to prove the following lemma.

\begin{lem}\label{lem-S2 S3m} Under GRH,
    $$\sum_{\ell \geq 1} S_{3\ell}(X) = W^{\ast}(X) \sum_{\ell \geq 1} \sum_p \frac{a(p)\log p}{p^{3\ell/2}} \phih\left(\frac{3\ell \log p}{L}\right) + O\left(X^{1/3}\right),$$
    $$S_2(X) \ll_{\epsilon} X^{1/2 + \epsilon},\qquad \sum_{\substack{m \not\equiv 0\text{ mod }3\\ m > 3}} S_m(X) \ll_{\epsilon} X^{1/2+\epsilon}, \qquad S_1(X) \ll_{\epsilon} X^{1/2 + \sigma/2 + \epsilon}.$$
\end{lem}
\begin{proof}
    Note that $\psi_{p^{3\ell}} = \psi_{p^3}$ for any $\ell \geq 1$. The first assertion follows directly from Lemma~\ref{lem Tag-5}. For $m \not\equiv 0$ mod $3$, we observe that $\psi_{p^m} = \psi_p$ or $\overline{\psi_p}$, both of which have conductor $\asymp p^2$. By Lemma~\ref{lem Tag-6} we have $$S_m(X) \ll_\epsilon \sum_{p} \frac{\log p}{p^{m/2}} \phih\left(\frac{m\log p}{L}\right) p^{\epsilon} X^{1/2+\epsilon},$$ and hence $$\sum_{\substack{m \not\equiv 0\text{ mod }3\\ m > 3}} S_m(X) \ll_\epsilon X^{1/2+\epsilon}.$$ We also have
    $$S_2(X) \ll_\epsilon \sum_p \frac{\log p}{p} \phih\left(\frac{2\log p}{L}\right) p^{\epsilon} X^{1/2+\epsilon} \ll_\epsilon X^{1/2 + \epsilon},$$ and 
    $$S_1(X) \ll_\epsilon \sum_p \frac{\log p}{\sqrt{p}} \phih\left(\frac{\log p}{L}\right) p^{\epsilon} X^{1/2 + \epsilon} \ll_\epsilon X^{1/2 + \sigma/2 + \epsilon}.$$
    
\end{proof}

\emph{Proof of Lemma~\ref{lem-third term of explicit formula}.} 
By Lemma~\ref{lem-S2 S3m}, it is now easy to see that Lemma~\ref{lem-third term of explicit formula} follows once we have 
    \begin{align}\label{eqn main of 3rd term}
    \sum_{\ell \geq 1}\sum_p \frac{a(p)\log p}{p^{3\ell/2}} \phih\left(\frac{3\ell\log p}{L}\right) &= \sum_{k=0}^{M-1} \left(\frac{\phih^{(k)}(0)}{k! L^k} \sum_{\ell\geq 1} \sum_p \frac{a(p) (3\ell)^k (\log p)^{k+1}}{p^{3\ell/2}}\right) \\
    &\qquad + O_{M,\phi}\left(\frac{1}{L^M}\right).\nonumber
    \end{align}
    Put $h(n) = a(n)\Lambda(n)$ where $\Lambda(n)$ is the Von Mangoldt function, and let $\eta = \frac{2(M+1)\log\log X}{\log X}$. For any fixed $M\geq 1$, we have
    \begin{align*}
        \phih(x) &= \infint \phi(\tau) e^{-2\pi i \tau x} d\tau\\
        &= \infint \phi(\tau) \left(\sum_{k=0}^{M-1} \frac{(-2\pi i \tau x)^k}{k!} + O\left(x^{M} \tau^{M}\right) \right) d\tau\\
        &= \sum_{k=0}^{M-1} \frac{\phih^{(k)}(0)}{k!} x^k + O_{M,\phi}\left(x^{M}\right).
    \end{align*}
    We thus write
    \begin{align*}
        \sum_{n=1}^{\infty} \frac{h(n)}{n^{3/2}} & \phih\left(\frac{3\log n}{L}\right)\\
        &= \sum_{n < X^{\eta}} \frac{h(n)}{n^{3/2}} \left(\sum_{k=0}^{M-1} \frac{\phih^{(k)}(0)}{k!} \frac{(3\log n)^k}{L^k} + O_{M,\phi}\left( \left(\frac{\log n}{L}\right)^{M}\right)\right) + O_{\phi}\left(X^{-\eta/2}\log \log X\right)\\
        &= \sum_{k=0}^{M-1} \frac{\phih^{(k)}(0)}{k! L^k} \sum_{n < X^{\eta}} \frac{h(n) (3\log n)^k}{n^{3/2}} + O_{M,\phi}\left(\frac{1}{L^{M}}\right)\\
        &= \sum_{k=0}^{M-1} \frac{\phih^{(k)}(0)}{k! L^k} \left(\sum_{n=1}^{\infty} \frac{h(n) (3\log n)^k}{n^{3/2}} + O\left(X^{-\eta/2} \left(\eta \log X\right)^{k+1}\right) \right) + O_{M,\phi}\left(\frac{1}{L^{M}}\right)\\
        &= \sum_{k=0}^{M-1} \frac{\phih^{(k)}(0)}{k! L^k} \sum_{n=1}^{\infty} \frac{h(n) (3\log n)^k}{n^{3/2}} + O_{M,\phi}\left(\frac{1}{L^{M}}\right),
    \end{align*}
which is \eqref{eqn main of 3rd term}.    
\qed

\subsection{The Fourth term of Lemma~\ref{lem explicit formula}: $\frac{4\pi}{L} \int_0^{\infty} \frac{e^{-\pi x}}{1 - e^{-4\pi x}} \left(\phih(0) - \phih\left(\frac{2\pi x}{L}\right) \right) dx$}\label{sec-fourth term}

\begin{lem}\label{lem fourth term}
    \begin{align*}
        \frac{4\pi}{L} \int_0^{\infty} &\frac{e^{-\pi x}}{1 - e^{-4\pi x}} \left(\phih(0) - \phih\left(\frac{2\pi x}{L}\right) \right) dx\\
        &\quad = \frac{\phih(0)}{L}\left(\frac{\pi}{2} + 3\log 2 + \gamma\right) 
        + \frac{1}{2L}\infint \phi(\tau) \left(\Psi\left(\frac{1}{4} - \frac{\pi i \tau}{L}\right) + \Psi\left(\frac{1}{4}+\frac{\pi i \tau}{L}\right)\right) d\tau.
    \end{align*}
\end{lem}
\begin{proof}
    As in the proof of \cite[Lemme I.2.1]{Mestre1986}, for a positive real number $\mathfrak{r}$, $$\int_0^{\infty} \frac{e^{-\mathfrak{r} x}}{1 - e^{-x}} \left(\phih(0) - \phih(ax)\right)dx = \frac{1}{2\pi} \infint \phi\left(\frac{t}{2\pi}\right) \left(\Psi(\mathfrak{r} + iat) - \Psi(\mathfrak{r})\right)dt.$$ Putting $\mathfrak{r} = 1/4$, $a = \frac{1}{2L}$, and using the special value $\Psi\left(\frac{1}{4}\right) = -\frac{\pi}{2} - 3\log 2 -\gamma,$ we write
    \begin{multline*}
    4\pi\int_0^{\infty} \frac{e^{-\pi x}}{1 - e^{-4\pi x}}\left(\phih(0) - \phih\left(\frac{2\pi x}{L}\right)\right) dx 
    = \infint \phi(\tau) \left(\Psi\left(\frac{1}{4} + \frac{\pi i \tau}{L}\right) - \Psi\left(\frac{1}{4}\right)\right)d\tau\\
    = \left(\frac{\pi}{2} + 3\log 2 + \gamma\right)\phih(0) + \infint \phi(\tau) \Psi\left(\frac{1}{4} + \frac{\pi i \tau}{L}\right)d\tau.
    \end{multline*}
    But $\phi(\tau)$ is an even function, and we can write $$\infint \phi(\tau) \Psi\left(\frac{1}{4} + \frac{\pi i \tau}{L}\right)d\tau = \frac{1}{2}\infint \phi(\tau) \left(\Psi\left(\frac{1}{4} - \frac{\pi i \tau}{L}\right) + \Psi\left(\frac{1}{4}+\frac{\pi i \tau}{L}\right)\right) d\tau.$$
\end{proof}

\begin{lem}\label{lem fourth term 2}
    For any $M \geq 1$,
    \begin{align*}
    \frac{4\pi}{L} \int_0^{\infty} &\frac{e^{-\pi x}}{1 - e^{-4\pi x}} \left(\phih(0) - \phih\left(\frac{2\pi x}{L}\right) \right) dx\\
    &\quad = -\sum_{k=1}^{M-1} \frac{\phih^{(k)}(0) 2^{k+2}\pi^{k+1}}{k! L^{k+1}} \int_0^{\infty} \frac{x^k e^{-\pi x}}{1-e^{-4\pi x}} dx + O_{M,\phi}\left(\frac{1}{L^{M+1}}\right).
    \end{align*}
\end{lem}
\begin{proof}
    Put $\eta = \frac{M}{\pi}\log \log X$ so that $e^{-\pi \eta} = (\log X)^{-M}$. Then
    \begin{align*}
        \int_0^{\infty} \frac{e^{-\pi x}}{1-e^{-4\pi x}} &\left(\phih(0) - \phih\left(\frac{2\pi x}{L}\right)\right) dx \\
        &= \int_0^{\eta} \frac{e^{-\pi x}}{1-e^{-4\pi x}} \left( -\sum_{k=1}^{M-1} \frac{\phih^{(k)}(0)(2\pi x)^k}{k! L^k} + O_{M,\phi}\left(\left(\frac{x}{L}\right)^M\right) \right) dx + O\left(e^{-\pi \eta}\right)\\
        &= -\int_0^{\eta} \frac{e^{-\pi x}}{1 - e^{-4\pi x}} \sum_{k=1}^{M-1} \frac{\phih^{(k)}(0)}{k!} \frac{(2\pi x)^k}{L^k} dx + O_{M,\phi}\left(\frac{1}{L^M}\right)\\
        &= -\sum_{k=1}^{M-1} \frac{\phih^{(k)}(0) 2^k \pi^k}{k! L^k} \int_0^{\eta} \frac{x^k e^{-\pi x}}{1-e^{-4\pi x}} dx + O_{M,\phi}\left(\frac{1}{L^M}\right)\\
        &= -\sum_{k=1}^{M-1} \frac{\phih^{(k)}(0) 2^k \pi^k}{k! L^k}  \int_0^{\infty} \frac{x^k e^{-\pi x}}{1-e^{-4\pi x}} dx + O_{M,\phi}\left( \frac{\eta^{M-1}}{L^{M+1}} + \frac{1}{L^M}\right),
    \end{align*}
    and we have Lemma~\ref{lem fourth term 2}.
\end{proof}

Now we can prove Theorem \ref{thm-main}. We replace the terms in Lemma~\ref{lem explicit formula} by Lemma~\ref{lem-third term of explicit formula} and Lemma~\ref{lem fourth term}. We can express the contribution of the third term as $$-\frac{2}{L}\sum_{\ell\geq 1} \sum_p \frac{a(p)\log p}{p^{3\ell/2}} \infint \phi(\tau) e^{-2\pi i \frac{3\ell \log p}{L}\tau} d\tau,$$ in which everything converges absolutely, and we can take the integral out of $\sum_{\ell}\sum_p$. Corollary~\ref{cor-main} directly follows from Lemma~\ref{lem explicit formula}, Lemma~\ref{lem-first term of explicit formula}, Lemma~\ref{lem-third term of explicit formula}, and Lemma~\ref{lem fourth term 2}.
\qed

\bigskip
\section{The Ratios Conjecture}\label{Sec_The Ratios Conjecture}

In this section, following the exposition of Conrey and Snaith \cite{Conrey2007}, we want to get a precise expectation of $\mathcal{D}^{\ast}(\phi;X)$ based on the Ratios conjecture for cubic Dirichlet $L$-functions. 

Let $f$ be an even Schwartz function. Following the convention of \cite[Section 3]{Conrey2007}, with $1/2 + 1/\log X < c < 3/4$, we let
\begin{align*}
S_X(f) &= \sum_{\alpha}''' w\left(\frac{N(\alpha)}{X}\right) \sum_{\gamma_{\alpha}} f(\gamma_{\alpha}) \\
&= \sum_{\alpha}''' w\left(\frac{N(\alpha)}{X}\right) \frac{1}{2\pi i}\left(\int_{(c)} - \int_{(1-c)}\right) \frac{L'(\chi_{\alpha},s)}{L(\chi_{\alpha},s)} f\left(-i\left(s-\frac{1}{2}\right)\right) ds.
\end{align*}
The integral over $(c)$ is
$$\frac{1}{2\pi} \infint f\left(t - i\left(c-\frac{1}{2}\right)\right)\sum_{\alpha}''' w\left(\frac{N(\alpha)}{X}\right) \frac{L'\left(\chi_{\alpha},\frac{1}{2} + \left(c - \frac{1}{2} + it\right)\right)}{L\left(\chi_{\alpha},\frac{1}{2} + \left(c - \frac{1}{2} + it\right)\right)} dt,$$ and assuming GRH, we have $L'/L \ll \log(N(\alpha)(|s|+3))$ \cite[(5.28)]{ANT2004}. Put $$R_w(\nu';\nu) = \sum_{\alpha}''' w\left(\frac{N(\alpha)}{X}\right) \frac{L\left(\chi_{\alpha},\frac{1}{2} + \nu'\right)}{L\left(\chi_{\alpha},\frac{1}{2} + \nu\right)},$$ so that 
\begin{equation}\label{eqn Tag-11}
\sum_{\alpha}''' w\left(\frac{N(\alpha)}{X}\right) \frac{L'\left(\chi_{\alpha},\frac{1}{2} + \left(c - \frac{1}{2}+it\right)\right)}{L\left(\chi_{\alpha},\frac{1}{2} + \left(c - \frac{1}{2} + it\right)\right)} = \frac{\partial}{\partial\nu'} R_w(\nu';\nu) \Biggr|_{\nu' = \nu = c - 1/2 + it}.
\end{equation}
When $n$ is a cube, under GRH, by Lemma~\ref{lem Tag-5} we have
\begin{equation*}
\sum_{\chi}^{\ast} w\left(\frac{q}{X}\right)\chi(n) = a(n) W^*(X) + \text{error}.
\end{equation*}
When $n$ is not a cube we assume $\sum_{\chi}^{\ast} w\left(\frac{q}{X}\right)\chi(n) = \text{small}$.

We expect the following conjecture to hold.
\begin{conj}[The Ratios conjecture for cubic Dirichlet $L$-functions]\label{conj Rw} 
    For $Re(2\nu'+\nu) > -1$ and $Re(\nu') > -1/2$, 
\begin{align*}
    R_w(\nu',\nu) &= \sum_{\alpha}''' w\left(\frac{N(\alpha)}{X}\right) \frac{L\left(\chi_{\alpha},\frac{1}{2} + \nu'\right)}{L\left(\chi_{\alpha},\frac{1}{2} + \nu\right)}\\
    &= W^*(X)  \left(\frac{\zeta\left(\frac{3}{2}+3\nu'\right)}{\zeta\left(\frac{3}{2}+2\nu'+\nu\right)} A(\nu';\nu) \right)
        + O\left(X^{1/2+\epsilon}\right),
\end{align*}
    where 
    \begin{align*}
    A(\nu';\nu) &= \frac{\zeta\left(\frac{3}{2}+2\nu'+\nu\right)}{\zeta\left(\frac{3}{2}+3\nu'\right)}\prod_p \left(1 + a(p)\frac{1-p^{\nu'-\nu}}{p^{3/2+3\nu'}-1}\right)  \\&= \prod_{p \equiv 1\mod 3}\left(1 + O\big(p^{-5/2 - 2\nu' - \nu} + p^{-5/2 - 3\nu'} + p^{-3-4\nu'-2\nu} +p^{-2}\big)\right).
    \end{align*}
\end{conj}

\emph{Derivation of Conjecture~\ref{conj Rw}.} 
We only consider the main term of the conjecture in this recipe. For the numerator of $R_w(\nu';\nu)$, we use the approximate functional equation. For a primitive cubic Dirichlet character $\chi$, \cite[Teopema 1]{Lavrik1968} becomes
\begin{equation}\label{eqn numerator}
L\left(\frac{1}{2}+\nu',\chi\right) 
= \sum_{n\leq x}\frac{\chi(n)}{n^{1/2+\nu'}} 
+ \epsilon(\chi)\left(\frac{q}{\pi}\right)^{-\nu'} \frac{\Gamma\left(\frac{1}{4}-\frac{\nu'}{2}\right)}{ \Gamma\left(\frac{1}{4} + \frac{\nu'}{2}\right)} \sum_{n\leq y} \frac{\overline{\chi}(n)}{n^{1/2-\nu'}}  + \text{Remainder},
\end{equation}
where $\epsilon(\chi)$ is the root number of $L(\chi,s)$.
For the denominator, as in \cite[Section 2.2]{Conrey2007} we write 
\begin{equation}\label{eqn denominator}
    \frac{1}{L(\chi,s)} = \sum_{h = 1}^{\infty} \frac{\mu(h)\chi(h)}{h^s}.
\end{equation}

Now we put \eqref{eqn numerator} and \eqref{eqn denominator} into the right hand side of Conjecture~\ref{conj Rw}. Summing up the contribution of the first term of \eqref{eqn numerator}, the only term that survives is
\begin{equation}\label{eqn Tag-13}
W^*(X) \sum_{hm = \cube} \frac{\mu(h)a(hm)}{h^{1/2+\nu} m^{1/2+\nu'}} = W^*(X) \prod_p \left(\sum_{\substack{h,m\\h + m \equiv 0\mod 3}} \frac{\mu(p^h) a(p^{h+m})}{p^{h(1/2+\nu) + m(1/2+\nu')}}\right).
\end{equation}
Assume $Re(\nu') > -1/2$. For $h = 0$, we have $\sum_{m \equiv 0\mod 3} \frac{a(p^m)}{p^{m(1/2+\nu')}} = 1 + \frac{a(p)}{p^{3/2 + 3\nu'} - 1},$ and for $h=1$,  
$$\sum_{m \equiv 2\mod 3} \frac{-a(p^{h+m})}{p^{1/2+\nu + m(1/2+\nu')}} = - \frac{a(p) p^{\nu' - \nu}}{p^{3/2 + 3\nu'} - 1}.$$ The term $\mu(p^h)$ removes the cases $h>1$, and the Euler factor at $p$ in \eqref{eqn Tag-13} becomes $$1 + a(p)\cdot \frac{1 - p^{\nu' - \nu}}{p^{3/2 + 3\nu'}-1}.$$
For $p \equiv 2$ mod $3$, this is $$1 + \frac{1 - p^{\nu'-\nu}}{p^{3/2+3\nu'}-1} = \frac{1 - p^{-3/2-2\nu'-\nu}}{1 - p^{-3/2-3\nu'}}.$$ Assume $Re(2\nu'+\nu) > -1$. For $p \equiv 1$ mod $3$, using the expression $\frac{1}{1 - \epsilon} = 1 + \epsilon + \epsilon^2 + \cdots$ repeatedly, we can write the Euler factor as
\begin{align*}
1 + &\frac{p - p^{1 + \nu' - \nu}}{(p+2)(p^{3/2 + 3\nu'}-1)}\\
&\qquad= 1 + \frac{p^{-3/2-3\nu'} - p^{-3/2 - 2\nu' - \nu}}{(1+2p^{-1})(1-p^{-3/2-3\nu'})}\\
&\qquad= \frac{(1 - 2p^{-1} + 4p^{-2} + \cdots)(1 + 2p^{-1} - p^{-3/2-2\nu'-\nu} - 2p^{-5/2 - 3\nu'})}{1 - p^{-3/2 - 3\nu'}}\\
&\qquad= \frac{1 - p^{-3/2-2\nu'-\nu} + O(p^{-2}+p^{-5/2-3\nu'} + p^{-5/2 - 2\nu' - \nu})}{1 - p^{-3/2 - 3\nu'}}\\
&\qquad= \frac{1 - p^{-3/2-2\nu'-\nu} + O(p^{-2} + p^{-5/2 - 3\nu'} + p^{-5/2-2\nu'-\nu}) }{1 - p^{-3/2-3\nu'}}\\
&\qquad \qquad \times (1 - p^{-3/2-2\nu'-\nu})(1 + p^{-3/2-2\nu'-\nu} + p^{-3-4\nu'-2\nu} + \cdots)\\
&\qquad = \frac{1 - p^{-3/2-2\nu'-\nu}}{1 - p^{-3/2-3\nu'}}\left(1+O(p^{-2}+p^{-3-4\nu'-2\nu} + p^{-5/2-3\nu'} + p^{-5/2-2\nu'-\nu})\right).
\end{align*}
Taking the product over all $p$, we get $$\zeta\left(\frac{3}{2}+3\nu'\right) \zeta\left(\frac{3}{2}+2\nu'+\nu\right)^{-1} A(\nu',\nu),$$ where $A(\nu',\nu) = \prod_{p\equiv 1\mod 3}\left(1 + O\big(p^{-5/2-2\nu'-\nu} + p^{-5/2 - 3\nu'} + p^{-3-4\nu'-2\nu} + p^{-2}\big)\right)$ converges absolutely for $Re(2\nu'+\nu) > -1$ and $Re(\nu') > -1/2$. 

In the Ratios Conjecture recipe, the root number $\epsilon(\chi)$ in the approximate functional equation is replaced by the expected value when averaged over the family. We assume that the average of the root numbers $\varepsilon(\chi)$ converges to $0$ as $q_\chi \rightarrow \infty$, hence the contribution from the dual series is zero and we get Conjecture~\ref{conj Rw}. This assumption is not an unlikely claim. Katz \cite{KatzGauss} showed that the root numbers $\varepsilon(\chi)$ for all Dirichlet characters are asymptotically equidistributed with respect to Haar measure on the unit circle in the complex plane in the limit as $q_\chi \rightarrow \infty$. Bob \cite{HoughJNT} obtained a result which is an extension of Katz' one. 
\qed

\subsection{Proof of Theorem \ref{thm-ratios conj prediction}}
We prove Theorem \ref{thm-ratios conj prediction} under Conjecture \ref{conj Rw}. 
To compute the right hand side of \eqref{eqn Tag-11}, we define $C_1(r)$ as follows:
\begin{equation}\label{eqn C1r}
\frac{\partial}{\partial \nu'} \left[ \zeta\left(\frac{3}{2}+3\nu'\right) \zeta\left(\frac{3}{2}+2\nu'+\nu\right)^{-1} A(\nu',\nu)\right]_{\nu' = \nu = r} =: C_1(r).
\end{equation}
Note that we can also write 
    \begin{align*}
C_1(z) 
&= \frac{\partial}{\partial \nu'} \left[ \prod_p \left(1 + a(p)\frac{1-p^{\nu'-\nu}}{p^{3/2+3\nu'}-1}\right)\right]_{\nu'=\nu=z}\\
&= \prod_p\left(1 + a(p)\frac{1-1}{p^{3/2+3z}-1}\right)\quad \cdot \quad \sum_q \left(1 + a(q)\frac{1-1}{q^{3/2+3z}-1}\right)^{-1}B(q,z)\\
&= \sum_p \frac{-a(p) \log p}{p^{3/2+3z}-1},
\end{align*}
where $\sum_q$ is over rational primes, and $$B(q,z) = \left[ \frac{a(q)\log q}{\left(q^{3/2+3\nu'}-1\right)^2} \left(-q^{\nu'-\nu+3/2+3\nu'} + q^{\nu'-\nu} -3q^{3/2+3\nu'} + 3q^{\nu'-\nu+3/2+3\nu'}\right)\right]_{\nu'=\nu=z}.$$ 
From the last expression we can write $$C_1(z) = - \sum_p \sum_{\ell \geq 1} \frac{a(p) \log p}{p^{3\ell/2 + 3\ell z}}.$$

Since  $\mathcal{D}^{\ast}(\phi;X) =S_X(f)/W^{\ast}(X)$ with $\hat{f}(x) = \frac{2\pi}{L}\phih(x)$, once we have the following propostion, Theorem \ref{thm-ratios conj prediction} follows. 

\begin{prop}\label{conj-Sf} Assume Conjecture \ref{conj Rw}.
    \begin{multline}
    S_X(f) = \frac{W^{\ast}(X)}{\pi} \infint f(t)  C_1(it)  dt
    + \frac{\hat{f}(0)}{2\pi} \sum_{\chi}^{\ast}w\left(\frac{q}{X}\right) \log \frac{q}{\pi} + C_2 W^{\ast}(X) + O\left(X^{1/2+\epsilon}\right),
    \end{multline}
    where $$C_2 = \frac{1}{4\pi} \infint f(-t) \left(\frac{\Gamma'}{\Gamma}\left(\frac{1}{4}-\frac{t}{2}i\right) + \frac{\Gamma'}{\Gamma}\left(\frac{1}{4} + \frac{t}{2}i\right)\right) dt.$$
\end{prop}
\begin{proof}
We first note that Cauchy's differentiation formula implies that the error term $O(X^{1/2+\epsilon})$ in $R_w(\nu';\nu)$, after taking the derivative, remains the same \cite[Lemma A.6]{FPS2016}. Writing $S_X(f) = \int_{(c)} - \int_{(1-c)}$, we shall have
\begin{equation}\label{eqn Tag-12}
\int_{(c)} = \frac{W^*(X)}{2\pi} \infint f\left(t - i\left(c-\frac{1}{2}\right)\right) C_1\left(c-\frac{1}{2} + it\right) dt + O\left(X^{1/2+\epsilon}\right).
\end{equation}
For $\int_{(1-c)}$ as in \cite[(3.8), (3.9)]{Conrey2007}, we write
\begin{multline*}
\int_{(1-c)} \frac{L'(\chi,s)}{L(\chi,s)} f\left(-i\left(s-\frac{1}{2}\right)\right)ds 
= \int_{(c)} \frac{L'(\chi,1-s)}{L(\chi,1-s)} f\left(i\left(s-\frac{1}{2}\right)\right)ds \\
= \int_{(c)} f\left(i\left(s-\frac{1}{2}\right)\right) \left( -\log \frac{q}{\pi} - \frac{1}{2} \frac{\Gamma'}{\Gamma}\left(\frac{1-s}{2}\right) - \frac{1}{2} \frac{\Gamma'}{\Gamma}\left(\frac{s}{2}\right) - \frac{L'(\chi,s)}{L(\chi,s)} \right)ds.
\end{multline*}
Since $f$ is even, we have
\begin{align*}
S_X(f) &= \frac{1}{2\pi i}\sum_{\chi}^{\ast} w\left(\frac{q}{X}\right) \left\{ \int_{(c)} f\left(-i\left(s-\frac{1}{2}\right)\right)\frac{L'(\chi,s)}{L(\chi,s)}ds \right. \\
& + \int_{(c)} f\left(i\left(s-\frac{1}{2}\right)\right) \left(\log \frac{q}{\pi} + \left. \frac{1}{2}\frac{\Gamma'}{\Gamma}\left(\frac{1-s}{2}\right) + \frac{1}{2}\frac{\Gamma'}{\Gamma}\left(\frac{s}{2}\right) + \frac{L'(\chi,s)}{L(\chi,s)}\right)ds \right\}\\
&= \frac{1}{\pi i} \sum_{\chi}^{\ast} w\left(\frac{q}{X}\right) \int_{(c)} f\left(-i\left(s-\frac{1}{2}\right)\right) \frac{L'(\chi,s)}{L(\chi,s)} ds 
+ \frac{1}{2\pi i} \sum_{\chi}^{\ast} w\left(\frac{q}{X}\right) \log \frac{q}{\pi} \int_{(c)} f\left(i\left(s-\frac{1}{2}\right)\right)ds\\ 
&\qquad\qquad+ \frac{1}{4\pi i} \sum_{\chi}^{\ast} w\left(\frac{q}{X}\right) \int_{(c)} f\left(i\left(s-\frac{1}{2}\right)\right) \left( \frac{\Gamma'}{\Gamma}\left(\frac{1-s}{2}\right) + \frac{\Gamma'}{\Gamma}\left(\frac{s}{2}\right)\right)ds.
\end{align*}

As for the first term here, by \eqref{eqn Tag-12} we can push the contour of integration $(c)$ to $c = 1/2$. The proposition follows immediately.
\end{proof}

\section{Appendix}
For reader's convenience, we record some known bounds on the value of $L$-functions and their logarithmic derivatives and properties of cubic residue symbol.  

\bigskip
Let $L(f,s)$ be an $L$-function in accordance with the definition of an $L$-function in \cite{ANT2004} and $q(f,s)$ be its analytic conductor. 
\begin{lem}\label{bound lemma1}
For $-1/2 \leq \mathfrak{r} \leq 2$, $\mathfrak{r} \neq 1/2$, as $|t| \rightarrow \infty$ we have, for $s=\mathfrak{r} + it$,
\begin{equation}\label{eqn_estimate of log derivatives}
\left|\frac{L'}{L}(f,s)\right| \ll \log q(f,s).
\end{equation}
\end{lem}
\begin{proof}
Assume that $t$ is sufficiently large. Then, by \cite[(5.28)]{ANT2004}, we have
\begin{eqnarray*}
\frac{L'}{L}(f, s) \ll  \left| \sum_{|s-\rho|<1}\frac{1}{s-\rho}\right| + \log q(f,s),
\end{eqnarray*}
where $\rho$'s are non-trivial zeros of $L(f,s)$. Let $m(T,f)$ be the number of zeros $\rho=\beta + i\gamma$ such that $|\gamma-T|\leq1$. Then, \cite[(5.27)]{ANT2004} says that $m(T,f) \ll \log q(f, iT)$. Hence, the Lemma follows. 
\end{proof}

\begin{lem}\label{bound lemma2}
\cite[Theorem 5.19]{ANT2004} Assume GRH and Ramanujan-Petersson conjecture for $L(f,s)$.
\begin{enumerate}
    \item For $\mathfrak{r} \geq 1/2+\epsilon$,
    \begin{equation}\label{eqn Tag-0}
    q(f,s)^{-\epsilon} \ll_\epsilon p_r(s) L(f,s) \ll_\epsilon q(f,s)^{\epsilon},
    \end{equation}
    where $r$ is the order of the pole at $s=1$ and $p_r(s) = (s-1)^r (s+1)^{-r}$.
    \item For $\mathfrak{r} = 1/4+\delta$, $0 \leq \delta < 1/4$, $q(\overline{f}) \asymp q(f)$,
    \begin{equation}\label{eqn Tag-0'}
    L\left(f,\frac{1}{4} + \delta + it\right) \ll_{\delta} q(f)^{1/4-\delta} q(f,s)^{\epsilon}.
    \end{equation}
\end{enumerate} 
\end{lem}
\begin{proof}
$(\ref{eqn Tag-0})$ is a consequence of \cite[Theorem 5.19]{ANT2004}.
The functional equation can be written $$L(f,1-s) = \varepsilon(f)\frac{q(\overline{f})^{s/2}}{q(f)^{\frac{1-s}{2}}} \frac{\gamma(\overline{f},s)}{\gamma(f,1-s)} L(\overline{f},s),$$ from which $(\ref{eqn Tag-0'})$ follows.
\end{proof}
\medskip
We say that an element $\alpha \in \mathbb{Z}[\omega]$ is \emph{primary} if $(\alpha,3)=1$ and $\alpha \equiv r$ mod $(1-\omega)^2$ for some $r \in \mathbb{Z}$. Equivalently, $\alpha$ is primary if $\alpha \equiv \pm 1$ mod $3$ in $\mathbb{Z}[\omega]$\footnote{There are 9 residue classes modulo 3.}. If $(\alpha,3)=1$ then one of $\alpha$, $\omega \alpha$, $\omega^2 \alpha$ is primary because there are 6 units in $\mathbb{Z}[\omega]$ which are all distinct modulo $3$. The product of primary numbers is again primary, and the complex conjugate of a primary number is also primary. Since $\mathbb{Q}(\sqrt{-3})$ has class number $1$, any number $\alpha \in \mathbb{Z}[\omega]$ has a unique factorization of the form $(-1)^s \omega^k (1-\omega)^h \pi_1^{e_1} \pi_2^{e_2} \cdots$ where each $\pi_i$ is congruent to $1$ modulo $3$, and $s,k,h,e_i \in \mathbb{Z}_{\geq 0}$. 

Let $\pi\in\mathbb{Z}[\omega]$ be a prime element coprime to 3. Observe that $N(\pi) \equiv 1$ mod 3. We collect well-known properties of the cubic residue symbol as follows.
\begin{prop}\label{prop 12 cubic residue symbol properties} \cite[Chap. 9]{IR}\cite[Chap. 7]{Lem}
    \begin{enumerate}
        \item $\left(\frac{\alpha}{\pi}\right)_3 \equiv \alpha^{\frac{N(\pi)-1}{3}} \equiv \omega^k\textbf{ mod }\pi$ for a unique $k \in \{0,1,2\}$.\label{eqn 13 defn of cubic residue symbol}
        \item If $\alpha \equiv \beta$ mod $\pi$ then $\left(\frac{\alpha}{\pi}\right)_3 = \left(\frac{\beta}{\pi}\right)_3$.
        \item $\left(\frac{\alpha\beta}{\pi}\right)_3 = \left(\frac{\alpha}{\pi}\right)_3 \left(\frac{\beta}{\pi}\right)_3$. ($\alpha,\beta$ need not be coprime to each other.)
        \item $\overline{\left(\frac{\alpha}{\pi}\right)}_3 = \left(\frac{\overline{\alpha}}{\overline{\pi}}\right)_3$.
        \item  If $\pi$ and $\theta$ are associates, $\left(\frac{\alpha}{\pi}\right)_3 = \left(\frac{\alpha}{\theta}\right)_3$. (In particular, $\left(\frac{\alpha}{\pi}\right)_3 = \left(\frac{\alpha}{-\pi}\right)_3$.)
        \item $x^3 \equiv \alpha$ mod $\pi$ has a solution $x \in \mathbb{Z}[\omega]$ if and only if  $\left(\frac{\alpha}{\pi}\right)_3 = 1$. In particular, $\left(\frac{-1}{\pi}\right)_3 = 1$.
        \item If $a,b\in\mathbb{Z}$ satisfy $(a,b) = (b,3)=1$ then $\left(\frac{a}{b}\right)_3=1$.\label{eqn property a b equals 1}
        \item The cubic residue symbol extends (in the denominator) into composite numbers by $$\left(\frac{\alpha}{\pi_1^{e_1}\pi_2^{e_2}\cdots}\right)_3 = \left(\frac{\alpha}{\pi_1}\right)_3^{e_1} \left(\frac{\alpha}{\pi_2}\right)_3^{e_2} \cdots.$$
        \item (Cubic Reciprocity) For $\alpha$, $\beta$ primary,  $\left(\frac{\alpha}{\beta}\right)_3 = \left(\frac{\beta}{\alpha}\right)_3$.\label{eqn property cubic reciprocity}\\
        
        \noindent Assume $\pi = 1+3a+3b\omega$ is primary with $a,b \in \mathbb{Z}$. (If $\pi \equiv -1$ mod $3$, replace $\pi$ by $-\pi$.) Then
        \item $\left(\frac{\omega}{\pi}\right)_3 = \omega^{2a+2b}$.
        \item $\left(\frac{1-\omega}{\pi}\right)_3 = \omega^a$.
        \item $\left(\frac{3}{\pi}\right)_3 = \omega^b$.		
    \end{enumerate}
\end{prop}

\end{document}